\documentclass[12pt]{article}
\usepackage[utf8]{inputenc}

\usepackage{amsmath,amssymb,amsthm} 
\usepackage{xcolor}
\usepackage[numbers,sort]{natbib}
\global\long\def\argmin{\operatorname*{argmin}}%

\definecolor{green}{rgb}{0,0.6,0.0}
\usepackage{amsmath}
\usepackage{amsfonts}
\usepackage{latexsym}
\usepackage{amssymb}
\usepackage{framed}
\usepackage{algorithm,algorithmic}

\newtheorem{thm}{Theorem}[section]

\newtheorem{lem}[thm]{Lemma}
\newtheorem{lemma}[thm]{Lemma}

\newtheorem{proposition}[thm]{Proposition}

\newtheorem{remark}[thm]{Remark}

\newcommand{\beq}{\begin{equation}}
\newcommand{\eeq}{\end{equation}}
\newcommand{\beqa}{\begin{eqnarray}}
\newcommand{\eeqa}{\end{eqnarray}}
\newcommand{\beqas}{\begin{eqnarray*}}
\newcommand{\eeqas}{\end{eqnarray*}}
\newcommand{\bi}{\begin{itemize}}
\newcommand{\ei}{\end{itemize}}

\newcommand{\vgap}{\vspace{.1in}}

\newcommand{\R}{\mathbb{R}}

\newcommand{\lam}{{\lambda}}

\newcommand{\inner}[2]{\langle #1,#2\rangle}

\newcommand{\dom}{\mathrm{dom}\,}

\newcommand{\tx}{\tilde x}

\title{FISTA and Extensions -- Review and New Insights}
\author{Weiwei Kong, Jefferson G. Melo, and Renato D.C. Monteiro}
\date{July 2021}

\begin{document}

\maketitle

\section{Introduction}

The purpose of this technical report is to review the main properties of
 an accelerated composite gradient (ACG) method commonly referred to
as the Fast Iterative Shrinkage-Thresholding Algorithm (FISTA). In addition, we state a version of FISTA for solving both convex and strongly convex
 composite minimization problems and derive its iteration
complexities to generate iterates satisfying
various stopping criteria, including one which arises in the course of
solving other composite optimization problems via inexact
proximal point schemes.
This report also discusses different reformulations of the convex version
of FISTA and how they relate to other formulations in the literature.

\vgap 

\noindent \textit{Organization}.
Section~\ref{sec:ACG} contains three  subsections. The first one describes a composite optimization problem and its main assumptions. The second subsection states and analyze a variant of FISTA, called S-FISTA, for solving the aforementioned problem. The third subsection establishes some iteration-complexity bounds for S-FISTA to obtain approximate stationary solution for the composite optimization problem we are interested in.  Section~\ref{Sec:Fista-mu=0} presents  an alternative formulation for S-FISTA and shows that it becomes the well-known FISTA for solving (non strongly) convex composite optimization problems.

\subsection{A Brief History of FISTA}

An earlier prototype of FISTA was given in \cite{nesterov1983method}, which proposed an ACG method named the Fast Gradient Method (FGM) for solving smooth convex (non-composite) optimization problems. FISTA, which is an extension of \cite{nesterov1983method} to smooth convex composite optimization problems, was then proposed in \cite{beck2009fast}. Its monotonically decreasing variant called M-FISTA was later proposed in \cite{beck2009fast_v2}. 

\section{A Strongly Convex Extension of FISTA}\label{sec:ACG}
This section contains three subsections. The first one describes a composite optimization problem and its main assumptions. The second subsection states and analyze a variant of FISTA, called S-FISTA, for solving the aforementioned problem. The third subsection establishes some iteration-complexity bounds for S-FISTA to obtain approximate stationary solution for the composite optimization problem we are interested in.

\subsection{Problem Description and Assumptions}

We are interested in the following problem
\begin{equation}\label{OptProb}
\phi^*:=\min\{ \phi(x):= f(x) + h(x) : x \in \R^n\}
\end{equation}
where $f: \R^n\to \R$ is a differentiable $\bar\mu_f$-convex function and $h:\R^n \to \R\cup\{+\infty\}$ is a possibly nonsmooth $\bar\mu_h$-convex function, and $\bar\mu_f, \bar\mu_h\geq 0$.

In addition, the following assumptions are made.

\begin{itemize}
\item[\textbf{(A)}] Problem \eqref{OptProb} has an optimal solution.

\item[\textbf{(B)}]
There exists a scalar $\bar{L}_f \geq \bar\mu_f $ such that 
\begin{equation}\label{ineq:uppercurvature}
f(\cdot)\leq l_f(\cdot,z)+\frac{\bar{L}_f}{2}\|\cdot - z\|^2,     \end{equation}
where
\begin{equation}\label{def:linearization f}
l_f(\cdot,z):= f(z)+\inner{\nabla f(z)}{\cdot -z}.    
\end{equation}
\end{itemize}

Clearly the following inclusion holds for any solution $x^*$ of \eqref{OptProb}:
$$
0\in \nabla f(x^*)+\partial h(x^*).
$$
For a given tolerance $\rho>0$, we say that a pair $(y,u)\in \R^n\times\R^n$ is a $\rho$-approximate stationary solution for problem \eqref{OptProb} if the following relations hold
\begin{equation}\label{def:approx-solution1}
u\in \nabla f(y)+\partial h(y), \quad \|u\|\leq \rho.    
\end{equation}

Next, we introduce a scalar which measures the distance of  the initial point $x_0$ to the solution set of \eqref{OptProb}.
\begin{equation}\label{def:d0}
d_0:= \min \{\|x_0-x^*\|: x^*\;\mbox{is a solution of}\; \eqref{OptProb}\}.    
\end{equation}
Recall that if a function $\psi:\R^n\to\R\cup\{+\infty\}$ is $\nu$-convex then, for every $z^*$ that minimizes $\psi$, we have
\begin{equation}\label{ineq:nu-convex}
\psi(z^*) +\frac{\nu}{2}\|\cdot - z^*\|^2\leq \psi(\cdot).     
\end{equation}
Moreover, since $f$ is $\bar\mu_f$-convex, the following inequality holds for any $z\in \R^n$:
\begin{equation}\label{ineq:muf-convex}
l_f(\cdot,z) +\frac{\bar\mu_f}{2}\|\cdot - z\|^2\leq f(\cdot),    
\end{equation}
where $l_f(\cdot,z)$ is as in \eqref{def:linearization f}.

Throughout this note we use the following notation $\log^+_1(\cdot):=\max\{\log(\cdot),1\}$.
\subsection{Statement and Properties of S-FISTA}

Recall that FISTA is a popular ACG variant for solving \eqref{OptProb}
for the case where $\bar{\mu}_f=\bar{\mu}_h=0$.
This subsection describes an extension of FISTA for solving
\eqref{OptProb} for the general case where $\bar{\mu}_f, \bar{\mu}_h \ge 0$.

We start by stating a strongly convex variant of FISTA, referred
to as S-FISTA,
for solving \eqref{OptProb}.

\begin{algorithm}[H]
	\caption{(S-FISTA)}
	\begin{algorithmic}
		\STATE 0. Let initial point $ x_0 \in \dom h$ and scalars $L_f > \bar{L}_f$, $\mu_f \in [0,\bar \mu_f]$ and
		$\mu_h \in [0,\bar \mu_h]$ be given,
		and set $ x_0=y_0 $, $ A_0=0 $, $\tau_0=1$, $\lam=1/(L_f - \mu_f)$, $\mu=\mu_f+\mu_h$, and $ k=0 $;
		\STATE 1. Compute
		\begin{equation}\label{def:ak-sfista}
		a_k=\frac{\lam\tau_k+\sqrt{(\lam\tau_k)^2+4\lam \tau_k A_k}}{2}, \quad A_{k+1}=A_k+a_k, \quad \tx_k=\frac{A_ky_k+a_kx_k}{A_{k+1}};
		\end{equation}
		\STATE 2. Compute
		\begin{align}
		y_{k+1}&:=\underset{x\in \dom h}\argmin\left\lbrace q_k^L (x;\tx_k) 
		:= \ell_f(x;\tilde x_k) + h(x) + \frac{L}{2}\|x-\tx_k\|^2\right\rbrace,
		\label{eq:ynext-sfista}\\
		\tau_{k+1} &= \tau_k + \mu a_k,  \label{eq:taunext-sfista} \\
		\quad x_{k+1}&= \frac{1}{\tau_{k+1}} \left[ \frac{a_k}{\lam}(y_{k+1}-\tx_k)+ \mu a_k y_{k+1} + \tau_k x_k \right] ; \label{eq:xnext-sfista}
		\end{align}
		\STATE 3. Set $ k \leftarrow k+1 $ and go to step 1.
	\end{algorithmic}
\end{algorithm}

We now make some comments about S-FISTA.
First, if $\mu=0$, then $\tau_k=1$ for every $k\geq 0$.
Second, the first and second relations in \eqref{def:ak-sfista} imply that
\begin{equation} \label{eq:Aalam-sfista}
 \frac{\tau_k  A_{k+1}}{ a_k^2}= \frac1\lam = L_f - \mu_f.
\end{equation}
Third, it will be shown in Section~\ref{Sec:Fista-mu=0} that when $\mu=0$, S-FISTA is actually FISTA.

Next, we  present some technical lemmas about S-FISTA.
\begin{lemma}\label{lem:gamma-sfista}
	For every $ k\ge 0 $ and $x \in \R^n$, define 
	\begin{align}
		\tilde \gamma_k(x)&:= \ell_f(x;\tilde x_k)  + h(x) + \frac{\mu_f}2 \|x-\tx_k\|^2, \label{def:tgamma-sfista} \\
		\gamma_k(x)&:= \tilde \gamma_k(y_{k+1}) + \frac{1}{\lam}\inner{\tx_k - y_{k+1}}{x - y_{k+1}}
		+ \frac{\mu}2 \|x-y_{k+1}\|^2. \label{def:gamma-sfista}
	\end{align}
	Then, the following statements hold for every $k \ge 0$:
	\begin{itemize}
		\item[a)] $ \tilde \gamma_k(y_{k+1}) = \gamma_k(y_{k+1}) $;
		\item[b)]
		 $\tilde \gamma_k \le \phi$ and
		\beq
		y_{k+1} = \argmin_{x}\left\{\tilde{\gamma}_{k}(x)+\frac{1}{2 \lam}\left\|x-\tilde{x}_{k}\right\|^{2}\right\} \label{eq:min-sfista'};
		\eeq
		\item[c)] $\gamma_k \le \tilde \gamma_k$ and
		\begin{equation}\label{eq:min-sfista}
		y_{k+1} = \min _{x}\left\{\gamma_{k}(x)+\frac{1}{2 \lam}\left\|x-\tilde{x}_{k}\right\|^{2}\right\};
		\end{equation}
		\item[d)] $x_{k+1}=\underset{x \in \R^n}\argmin\left\{a_{k} \gamma_{k}(x)+\tau_k \left\|x-x_{k}\right\|^{2} /2 \right\}$
		and
		\[
		x_{k+1} = x_k + \frac{a_k}{\tau_{k+1}} \left[ \frac{1}{\lam}(y_{k+1}-\tx_k)+ \mu ( y_{k+1} - x_k) \right];
		\]
\item[e)] $\tau_k = 1+A_k \mu$.
	\end{itemize}
\end{lemma}
\begin{proof}
	a) It clearly follows from \eqref{def:gamma-sfista} that $ \gamma_k(y_{k+1}) = \tilde \gamma_k(y_{k+1}) $.
	
	b) The inequality $\tilde \gamma_k \le \phi$ follows from the definition of $\phi$ in \eqref{OptProb}, \eqref{ineq:muf-convex}, and \eqref{def:tgamma-sfista}.
	Moreover, \eqref{eq:min-sfista'} follows from \eqref{eq:ynext-sfista}, \eqref{def:tgamma-sfista}, and the fact
	that $\lam=1/(L_f - \mu_f)$.

	c) Define $\hat \gamma_k := \tilde \gamma_k - \mu\|\cdot-y_{k+1}\|^2/2$. Since $\tilde \gamma_k$ is $\mu$-convex,
	it follows that $\hat \gamma_k$ is convex, and hence that $\partial \tilde \gamma_k(y_{k+1})=\partial \hat \gamma_k(y_{k+1})$,
	in view of the subgradient rule for the sum of two convex functions.
	Also, the optimality condition for \eqref{eq:min-sfista'} implies that
	\[
	\frac{\tx_k-y_{k+1}}{\lam_k} \in \partial \tilde \gamma_k(y_{k+1}) =
	\partial  \hat \gamma_k   (y_{k+1}),
	\]
	which, in view of the definition of $\hat \gamma_k$ and its subgradient at $y_{k+1}$, is easily seen to be equivalent to
	$\gamma_k \le \tilde \gamma_k$. Now, since $\nabla \gamma_k(y_{k+1})=(\tx_k-y_{k+1})/\lam$ by \eqref{def:gamma-sfista},
	we easily see that
	$y_{k+1}$ satisfies the optimality condition for \eqref{eq:min-sfista}, and hence  \eqref{eq:min-sfista} in follows.
	
%

	d) It follows from \eqref{def:gamma-sfista}, \eqref{eq:taunext-sfista} and \eqref{eq:xnext-sfista} that
	\begin{align*}
	a_k \nabla \gamma_k(x_{k+1})  + \tau_k (x_{k+1}-x_k) &=
	\left[ \frac{a_k}{\lam} (\tilde x_k-y_{k+1}) + a_k \mu ( x_{k+1}-y_{k+1} ) \right] + \tau_k (x_{k+1}-x_k) \\
&=  \frac{a_k}{\lam} (\tilde x_k-y_{k+1}) +  \tau_{k+1}  x_{k+1} - a_k \mu y_{k+1}  - \tau_k x_k = 0,
	\end{align*}
	and hence that the first claim in d) follows. The second claim follows  from \eqref{eq:taunext-sfista} and \eqref{eq:xnext-sfista}.
	 
	e) This identity follows follows immediately from \eqref{eq:taunext-sfista} and the second identity in \eqref{def:ak-sfista}.
\end{proof}

\vgap

\begin{lemma} \label{lm:hysub-2-sfista}
For every $k \ge 0$ and $x \in \dom h$, we have
\[
q^L_k(x;\tx_k)  \ge \phi(x) + \frac12 \left( L_f  - \bar{L}_f \right) \|x - \tx_k \|^2  
\]
where $q_k^{L}(\cdot;\cdot)$ is defined in \eqref{eq:ynext-sfista}.
\end{lemma}

\begin{proof}
Using the definitions of  $\phi$ and $q^L_k(\cdot;\tx_k)$ in  \eqref{OptProb} and \eqref{eq:ynext-sfista}, respectively, and inequality \eqref{ineq:uppercurvature}, we have 
\begin{align*}
q^L_k(x;\tx_k)&= \left( \ell_f(x;\tilde x_k) + \frac{\bar{L}_f}2 \|x-\tx_k\|^2 \right)+ h(x) + \frac12 \left( L_f  - \bar{L}_f \right) \|x-\tx_k\|^2 \\
&\ge \phi(x) + \frac12 \left( L_f  - \bar{L}_f \right) \|x - \tx_k \|^2
\end{align*}
for every $x \in \R^n$.
\end{proof}

\begin{lemma} \label{lm:hysub-3-sfista}
	For every $ k\ge 0 $ and $x \in \R^n$, we have
	\begin{align}\label{ineq:recur}
	A_k\gamma_k(y_k) + a_k\gamma_k(x) + \frac12 \|x_k - x\|^2 - \frac12 \|x_{k+1} - x\|^2
	\ge A_{k+1} q_k^{L}(y_{k+1};\tx_k) .
	\end{align}
	where $\gamma_k(\cdot)$ and $q_k^{L}(\cdot;\cdot)$ are defined in \eqref{def:gamma-sfista} and
	\eqref{eq:ynext-sfista}, respectively.
\end{lemma}

\begin{proof}
Using Lemma \ref{lem:gamma-sfista}(d), the facts  that $\tau_{k+1}=\tau_k+\mu a_k$ (see step 3 of Algorithm 1) and $\psi_k:=a_k\gamma_k(\cdot)+\tau_k \|\cdot-\tx_k\|^2/2$ is  $(\tau_k+\mu a_k)$-convex, it follows from 
\eqref{ineq:nu-convex} with $\psi=\psi_k$ and $\nu=\tau_{k+1}$ that 
\begin{align*}
a_k\gamma_k(x) + \frac{\tau_k}2 \|x-x_k\|^2 - \frac{\tau_{k+1}}2 \|x-x_{k+1}\|^2 \ge a_k\gamma_k(x_{k+1}) + \frac{\tau_k}2 \|x_{k+1}-x_k\|^2 \quad
\forall x \in \R^n.
\end{align*}
	Using the convexity of $ \gamma_k $, the definitions of $A_{k+1}$ and  $ \tx_k $ in \eqref{def:ak-sfista}, and relation \eqref{eq:Aalam-sfista}, we have
	\begin{align*}
		A_k\gamma_k(y_k) &+ a_k\gamma_k(x_{k+1}) + \frac{\tau_k}2 \|x_{k+1}-x_k\|^2 \\
		&\ge A_{k+1} \gamma_k\left( \frac{A_ky_k+a_kx_{k+1}}{A_{k+1}} \right) + \frac{\tau_kA_{k+1}^2}{2a_k^2}\left\| \frac{A_ky_k+a_kx_{k+1}}{A_{k+1}}- \frac{A_ky_k+a_kx_{k}}{A_{k+1}} \right\| ^2\\
		&= A_{k+1} \left[ \gamma_k\left( \frac{A_ky_k+a_kx_{k+1}}{A_{k+1}} \right) + \frac{1}{2\lam} \left\| \frac{A_ky_k+a_kx_{k+1}}{A_{k+1}}-\tx_k\right\| ^2\right] \\
		&\ge A_{k+1}\min_{x}\left\lbrace \gamma_k(x) + \frac{1}{2\lam}\|x-\tx_k\|^2\right\rbrace \\
		&= A_{k+1}\left[ \tilde \gamma_k(y_{k+1}) + \frac{L_f - \mu_f}{2}\|y_{k+1}-\tx_k\|^2\right] = A_{k+1}q_k^L(y_{k+1};\tx_k) 
	\end{align*}
	where the second last equality is due to Lemma \ref{lem:gamma-sfista}(b) and the fact that  $\lam^{-1}=L_f - \mu_f$,
	and the last one is due to \eqref{def:tgamma-sfista} and the definition of
	$q_k^{L}(\cdot;\cdot)$ in \eqref{eq:ynext-sfista}.
	The lemma now follows by combining the above two conclusions.
	\end{proof}
The next two results provide some important recursive formulas.	
\begin{lemma} \label{lm:hysub-30-fista}
	For every $ k\ge 0 $ and $x \in \R^n$, we have
	\begin{align}\label{ineq:recur2}
	A_k \phi(y_k) &+ a_k\gamma_k(x) + \frac{\tau_k}2 \|x_k - x\|^2 - \frac{\tau_{k+1}}2 \|x_{k+1} - x\|^2 \\
	&\ge A_{k+1} \phi(y_{k+1}) + \frac{A_{k+1}}2 \left( L_f - \bar{L}_f \right) \| y_{k+1} - \tx_k \|^2 .
	\end{align}
	where 
	$\gamma_k(\cdot)$ is  defined in \eqref{def:gamma-sfista}.
\end{lemma}

\begin{proof}
The conclusion of this result follows immediately from Lemma \ref{lm:hysub-2-sfista} with $x=y_{k+1}$ and Lemma \ref{lm:hysub-3-sfista}.
\end{proof}

\begin{lemma} \label{lm:hysub-4-fista}
For every $k \ge 0$ and $x \in \dom h$, we have
\[
\eta_k(x) - \eta_{k+1} (x) \ge   \frac{A_{k+1}}2 \left( L_f - \bar{L}_f \right) \|\tilde y_{k+1} - \tx_k \|^2
\]
where
\[
\eta_k(x) := A_k [ \phi(y_k) - \phi(x) ] + \frac{\tau_k}{2} \|x-x_k\|^2.
\]
\end{lemma}

\begin{proof}
Using Lemma \ref{lm:hysub-3-sfista} and the fact that  $\gamma_k \le \phi$ by Lemma \ref{lem:gamma-sfista}(b)-(c),  we have
\begin{align*}
A_k \phi (y_k) + & a_k\phi (x) + \frac{\tau_k}2 \|x_k - x\|^2 - \frac{\tau_{k+1}}2 \|x_{k+1} - x\|^2 \\
	& \ge  A_{k+1} \phi(y_{k+1}) + \frac{A_{k+1}}2 \left( L_f - \bar{L}_f \right) \|\tilde y_{k+1} - \tx_k \|^2 .
\end{align*}
The conclusion of the lemma now follows by subtracting $A_{k+1} \phi(x)$ from both sides of the above inequality, and using
the identity $A_{k+1}=A_k+a_k$ and the definition of $\eta_k(x)$.
\end{proof}

Next, we state a basic result that will be useful in deriving complexity bounds for S-FISTA.
\begin{lemma} \label{lm:hysub-5-ac}
For every $k \ge 0$ and $x \in \dom h$, we have
\begin{align*}
 A_k [ \phi(y_k) - \phi(x) ] + \frac{\tau_k}2 \|x-x_k\|^2 \le
 \frac12 \|x-x_0\|^2 - \frac12 \left( L_f - \bar{L}_f \right) \sum_{i=0}^{k-1} A_{i+1} \|\tilde y_{i+1} - \tx_i \|^2.
\end{align*}
\end{lemma}
\begin{proof}
This result follows  by summing  the inequality of Lemma \ref{lm:hysub-4-fista} from $k=0$ to $k=k-1$, and using
 the fact that $A_0=0$ and the definition of $\eta_k(\cdot)$ in Lemma \ref{lm:hysub-4-fista}.
\end{proof}

The below result gives some estimates on the sequence $\{A_k\}$.
\begin{lemma} \label{lm:Ak-est-fista}
For every $k \ge 1$, we have
\beq
A_k \ge \frac{1}{L_f - \mu_f} \max \left\{ 
\frac{ k^2}{4} ,  
\left(  1 + \frac12 \sqrt{ \frac{\mu}{L_f - \mu_f} } \right)^{2(k-1)} \label{eq:Akest-sfista}
 \right\}.
\eeq
As a consequence, 
for a given $\bar A>0$, we have $A_k \geq \bar{A}$ as long as 
\begin{equation}
k\geq \min \left\{2\sqrt{(L_f - \mu_f)\bar{A}},\;\left[\frac{1}{2}+\sqrt{\frac{L_f - \mu_f}{\mu}}\right]\log^+_1((L_f - \mu_f)\bar{A})+1\right\}.
\end{equation}

\end{lemma}
\begin{proof}
The first and second  identities in \eqref{def:ak-sfista} imply that
\[
A_{k+1} = A_k+ a_k \ge A_k +  \left( \frac{\tau_k \lam}2 + \sqrt{\tau_k \lam A_k} \right) \ge 
\left(  \sqrt{A_k} + \frac12 \sqrt{\tau_k \lam} \right)^2
\]
which, together with the fact that $\tau_k=1+\mu A_k$, yields
\[
\sqrt{A_{k+1}} \ge \sqrt{A_k} + \frac12 \sqrt{\tau_k \lam} = \sqrt{A_k} + \frac12 \sqrt{(1+\mu A_k) \lam}.
\]
Clearly, the last inequality implies the two inequalities
\[
\sqrt{A_{k+1}} \ge  \sqrt{A_k} + \frac12 \sqrt{ \lam}, \quad \sqrt{A_{k+1}} \ge \sqrt{A_k} \left(  1 + \frac12 \sqrt{\mu \lam} \right).
\]
The first bound in \eqref{eq:Akest-sfista} follows by summing the first  inequality from $k=0$ to $k=k-1$, and using the fact that $A_0=0$
and $\lam=1/(L_f - \mu_f)$ (see step 0 of S-FISTA). The second bound in \eqref{eq:Akest-sfista} follows by successively using  the second inequality
from $k=1$ to $k=k-1$ and using the fact that $A_1=\lam$.

Now to prove the last statement of the lemma note that 
\eqref{eq:Akest-sfista} implies that in order to have $A_k\geq \bar{A}$, it is sufficient to have 
$$\frac{1}{L_f - \mu_f} \max \left\{ 
\frac{ k^2}{4} ,  
\left(  1 + \frac12 \sqrt{ \frac{\mu}{L_f - \mu_f} } \right)^{2(k-1)}
 \right\}\geq \bar{A} .
$$
Clearly, the above condition is satisfied if one of the following conditions holds
$$
k\geq 2\sqrt{(L_f - \mu_f)\bar{A}}, \qquad 
\left(1 + \frac{1}{2} \sqrt{\frac{\mu}{L_f - \mu_f}} \right)^{2(k-1)}\geq (L_f - \mu_f)\bar{A}.
$$
The latter inequality is equivalent to
$$
2(k-1)\log\left(1 + \frac{1}{2} \sqrt{\frac{\mu}{L_f - \mu_f}} \right)\geq \log\left((L_f - \mu_f)\bar{A}\right).
$$
Since $\log(1+x)\geq 1/(1+x^{-1})$ for any $x>0$, it follows by using $x=\sqrt{\mu}/[2\sqrt{L_f - \mu_f}]$ that the above condition holds if 
$$
2(k-1)\left(\frac{1}{1+\frac{2\sqrt{L_f - \mu_f}}{\sqrt{\mu}}}\right)\geq \log^+_1\left((L_f - \mu_f)\bar{A}\right)
$$
which immediately proves the last statement of the lemma.
\end{proof}

The below result establishes  a convergence rate and iteration-complexity bounds for S-FISTA to obtain a approximate (function value) solution of \eqref{OptProb}.

\begin{proposition}
For every $k \ge 1$, we have
\begin{equation}\label{eq:convrate-phi(yk)}
\phi(y_k) - \phi^* \le \frac{(L_f - \mu_f) d_0^2}2  \min \left \{ \frac{4}{k^2} \,, \, 
 \left(  1 + \frac12 \sqrt{ \frac{\mu}{L_f - \mu_f} } \right)^{2(1-k)} \right\}
\end{equation}
where $\phi^*$ and $d_0$ are as in \eqref{OptProb} and \eqref{def:d0}, respectively.
As a consequence, for any given $\bar \varepsilon >0$, S-FISTA finds a point $y:=y_k$ satisfying $\phi(y)-\phi^* \leq \bar{\varepsilon}$
in at most
\[
{\cal O} \left( 
\min \left \{ 
d_0 \sqrt{ \frac{ L_f - \mu_f}{\bar \varepsilon} } \,,\, \sqrt{\frac{L_f - \mu_f}{\mu} } \log^+_1 \left( \frac{(L_f - \mu_f) d_0^2}{\bar \varepsilon} \right)  
 \right\} \right)
\]
iterations.
\end{proposition}
\begin{proof}
It follows from Lemma~\ref{lm:hysub-5-ac} with $x^*$ such that $d_0=\|x_0-x^*\|$ that,
for every $k \ge 0$, 
\begin{align*}
  \phi(y_k) - \phi^*\le
 \frac{1}{2A_k} d_0^2.
\end{align*}
Hence, \eqref{eq:convrate-phi(yk)}  follows immediately from \eqref{eq:Akest-sfista}. The last statement of the proposition follows immediately from the latter inequality  and the last statement of Lemma~\ref{lm:Ak-est-fista} with $\bar{A}= d_0^2/(2\bar\varepsilon).$
\end{proof}

\subsection{Stationarity Complexity Bounds}
This subsection is devoted to the study of iteration-complexity bounds for S-FISTA to compute several different notions of an approximate \textit{}{stationary} solution of \eqref{OptProb}.

We start by establishing an iteration-complexity bound for S-FISTA to obtain an approximate stationary solution of \eqref{OptProb} based on the generalized subdifferential of $\phi$.

\begin{lemma}
Assume that $\nabla f$ is $L$-Lipschitz continuous and define
\[
u_{k} = \nabla f(y_{k}) - \nabla f(\tx_{k-1}) + L_f ( \tx_{k-1}-y_{k}).
\]
Then, the following statements hold:
\begin{itemize}
    \item [a)]
for every $k \ge 1$,
\begin{equation}\label{eq:incl-ineq-uk}
u_k \in \nabla f(y_k) + \partial h(y_k), \quad \min_{1\leq i\leq k} \|u_i\|^2 \le \frac{8L_f^2d_0^2}{(L_f-\bar{L}_f)\sum_{i=1}^kA_i};
\end{equation}
\item[b)] for any $\rho>0$, S-FISTA generates  a $\rho$-approximate stationary solution pair $(y,u):=(y_k,u_k)$ in at most
$$
\left\lceil \min\left\{\left(\frac{12\zeta d_0^2}{\rho^2}\right)^{1/3}, \left(1
+\frac{
2\sqrt{L_f-\mu_f}}{\sqrt{\mu}}\right)\log\left(1+\frac{\zeta(c^2-1)d_0^2}{\rho^{2}}\right)\right\}\right\rceil
$$
iterations, where 
\begin{equation}\label{def:zeta}
\zeta=\zeta(\mu_f,L_f,\bar{L}_f):=\frac{8L_f^2(L_f-\mu_f)}{L_f-\bar{L}_f}, \qquad c=c(\mu_f,\mu,L_f)= 1+\frac{1}{2}\sqrt{\frac{\mu}{L_f-\mu_f}}. 
\end{equation}
\end{itemize}
\end{lemma}

\begin{proof}
a) It follows from \eqref{eq:ynext-sfista} with $k=k-1$ and its associated optimality condition that
\[
0 \in \nabla f(\tx_{k-1}) + \partial h(y_k)+  L_f ( y_k- \tx_{k-1})
\]
which,  in view of the definition of $u_k$,  immediately implies the inclusion of the lemma.
Using the definition of $u_k$, assumption that $\nabla f$ is $L_f$-Lipschitz continuous on $\R^n$, and the triangle inequality
for norms, it follows that
\[
\|u_k\| \le \| \nabla f(y_{k}) - \nabla f(\tx_{k-1}) \| + L_f  \|y_k- \tx_{k-1}\| \le 2 L_f \| y_k - \tx_{k-1}\|.
\]
Now using Lemma \ref{lm:hysub-5-ac} with $x=x^*$ where $d_0=\|x_0-x^*\|$, we conclude that
\[
d_0^2 \ge \frac12 \left( L_f - \bar{L}_f \right) \sum_{i=1}^{k} A_{i} \|\tilde y_{i} - \tx_{i-1} \|^2
\ge \frac{L_f-\bar{L}_f}{8L_f^2}  \sum_{i=1}^{k} A_{i} \|u_i\|^2,
\]
and hence that the  statement in (a)  holds.

b) First note that in view of a) the inclusion in  \eqref{def:approx-solution1} holds with for any $(y,u):=(y_k,u_k)$. Now, recall that
$$
\sum_{i=1}^k i^2=\frac{k(k+1)(2k+1)}{6} \geq \frac{k^3}{3}, \qquad \sum_{i=1}^{k}c^{2(i-1)}= \frac{c^{2k}-1}{c^2-1}
$$
for any nonzero  scalar $c \neq \pm 1$.
Hence, considering  $c$ as in \eqref{def:zeta}, it follows from  the above relations  and  \eqref{eq:Akest-sfista} that
\begin{align*}
\sum_{i=1}^k A_i \geq  
\sum_{i=1}^k\frac{1}{L_f-\mu_f}\max \left\{ \frac{ i^2}{4} ,  
c^{2(i-1)}\right\}\geq 
\frac{1}{L_f-\mu_f}\max \left\{ 
\frac{ k^3}{12} ,  
\frac{c^{2k}-1}{c^2-1}
 \right\}
 \end{align*}
 which combined with  \eqref{eq:incl-ineq-uk} and $\eqref{def:zeta}$ implies that
$$
\min_{1\leq i\leq k} \|u_i\|^2 \le \frac{8L_f^2(L_f-\mu_f)d_0^2}{L_f-\bar{L}_f} \min \left\{ 
\frac{ 12}{k^3},  
\frac{c^2-1}{c^{2k}-1}
 \right\}=\zeta d_0^2\min \left\{ 
\frac{ 12}{k^3},  
\frac{c^2-1}{c^{2k}-1}
 \right\}.
$$
Hence,  in order to obtain $\min_{1\leq i\leq k} \|u_i\|\leq \rho$, it is sufficient to have
$$
\zeta d_0^2\min \left\{ 
\frac{ 12}{k^3},  
\frac{c^2-1}{c^{2k}-1}
 \right\}\leq \rho^2,
$$
or equivalently, one of the following inequalities should hold
\begin{equation}\label{aux:k3-c}
k\geq \left(\frac{12\zeta d_0^2}{\rho^2}\right)^{1/3}, \quad \frac{\zeta d_0^2\left(c^2-1\right)}{\rho^2}
\leq c^{2k}-1.
\end{equation}
Note that the latter inequality is equivalent to 
$$
\frac{\log\left(1+\frac{\zeta d_0^2(c^2-1)}{\rho^{2}}\right)}{2\log c}
\leq k.
$$
Hence, since $\log(1+x)\geq 1/(1+x^{-1})$ for any $x>0$, it follows by using $x=\sqrt{\mu}/(2\sqrt{L_f-\mu_f})$ and the definition of $c$ that the above condition holds if 
$$
\left(1
+\frac{
2\sqrt{L_f-\mu_f}}{\sqrt{\mu}}\right)\log\left(1+\zeta d_0^2(c^2-1)\rho^{-2}\right)
\leq k.
$$
Hence, the last statement of the lemma follows from the above conclusion, the first inequality in  \eqref{aux:k3-c}, and the definition of $\zeta$ in \eqref{def:zeta}.
\end{proof}


Before discussing some more exotic notions of approximate solutions, we first establish some properties regarding $\Gamma_k$ and its relation to $\phi$.

\begin{lemma} \label{lm:hysub-31-sfista}
Define $\Gamma_k: \R^n \to \R$ as
	\begin{equation}\label{def:Gammak}
		\Gamma_k(x) : = \frac1{A_k} \sum_{i=0}^{k-1} a_i\gamma_i(x) \quad \forall x \in \R^n.
		\end{equation}
	Then, for every $k \ge 1$, the following statements hold:
	\begin{itemize}
	\item[a)] $\Gamma_k \le \phi$ and $\Gamma_k$ is a $\mu$-convex quadratic function with Hessian equal to $\mu I$;
	\item[b)] for every $x \in \R^n$, we have
	\begin{align}\label{ineq:Gamma_recur}
   \Gamma_k(x) 
	\ge \phi(y_{k}) + \frac1{2A_k} \left( \tau_k \|x_{k} - x\|^2 - \|x_{0} - x\|^2 \right) 
	\end{align}
	where $\gamma_k(\cdot)$ is  defined in \eqref{def:gamma-sfista};
\end{itemize}
\end{lemma}
\begin{proof}
a) In view of \eqref{def:gamma-sfista}, the above definition of $\Gamma_k$, (b) and (c) of  Lemma \ref{lem:gamma-sfista}, and the second relation in \eqref{def:ak-sfista}, it follows that
$\Gamma_k$ is a convex combination of a $\mu$-convex quadratic functions minorizing $\phi$ and whose  Hessian are all equal to $\mu I$.
Hence,  a) follows.

b) This statement follows by summing the inequality in Lemma \ref{lm:hysub-30-fista} from $k=0$ to $k=k-1$, using the definition of $\Gamma_k$,
and the fact that $A_0=0$ and $\tau_0=1$ (see step 0 of S-FISTA).
\end{proof}

The next result  shows some important relations on the pair $(v_k,\eta_k)$ defined below in \eqref{def:vk-etak}. This pair of elements can be incorporated in S-FISTA in order to apply it to inexactly solve some proximal subproblems.    
\begin{lemma} \label{lm:stat-sfista}
Define
\begin{equation}\label{def:vk-etak}
 	v_k := \mu (y_k-x_k) + \frac{x_0-x_k}{A_k}, \quad \eta_k := \frac1{2A_k} \left( \|x_{0} - y_k\|^2 - \tau_k \|x_{k} - y_k\|^2 \right).   
\end{equation}

Then, the following statements hold for every $k \ge 1$:
\begin{itemize}
\item[a)] for every $x \in \R^n$, we have
	\beq \label{eq:cruuu-sfista}
\Gamma_k(x) - \frac{\mu}{2} \|x-y_k\|^2
	\ge  \phi(y_{k}) + \inner{v_k}{x-y_{k}} - \eta_{k},
	\eeq
	where $\Gamma_k$ is as in \eqref{def:Gammak}.
\item[b)] we have 
	\begin{align*}
	&\eta_k \ge 0, \quad v_k \in \partial_{\eta_k} \left(  \phi - \frac{\mu}2 \|\cdot-y_k\|^2 \right) (y_k), \\
	& \frac{1}{\tau_k} \|A_k v_k + y_k - x_0\|^2 + 2A_k \eta_k = \|y_k-x_0\|^2;
	\end{align*}
\item[c)] we have
	\[
	\|v_k\| \le \frac{1+\sqrt{\tau_k}}{A_k} \|y_k-x_0\|, \quad \eta_k \le \frac{\|y_k-x_0\|^2}{2A_k}.
	\]
\end{itemize}	
\end{lemma}

\begin{proof}
a) This statement follows from Lemma \ref{lm:hysub-31-sfista}(b) and the fact that
the definitions of $v_k$ and $\eta_k$ combined with the relation in Lemma~\ref{lem:gamma-sfista}(e) imply that
\[
\frac1{2A_k} \left( \tau_k \|x_{k} - x\|^2 - \|x_{0} - x\|^2 \right) = \inner{v_k}{x-y_{k}} - \eta_{k}
+ \frac{\tau_k-1}{2A_k} \|x-y_k\|^2.
\]

b) In view of  Lemma~\ref{lm:hysub-31-sfista}(a), the  inequality and the inclusion in (b) follow from  the inequality in (a), first with  $x=y_k$ and then with arbitrary  $x\in \R^n$. The last relation in (b) follows from the definitions of $v_k$ and $\eta_k$ combined with Lemma~\ref{lem:gamma-sfista}(e).

c) These inequalities follow immediately from the last relation in (b) together with the triangle inequality for norms.
\end{proof}
The next result shows how the sequence $\{y_k\}$ together with the residuals pair sequence  $(v_k,\eta_k)$ defined in \eqref{def:vk-etak} can be used to generate an approximate solution based on a relative error criterion. An iteration complexity bound is also given for convenience.

\begin{lemma} \label{lem:fista_basic_compl}
Let $\{y_k\}$ be generated by S-FISTA and let $\{(v_k,\eta_k)\}$ defined as in \eqref{def:vk-etak}. Then, for any $\tilde\sigma>0$ and $k \geq 1$, the  triple $(y,v,\eta):=(y_k,v_k,\eta_k)$  satisfies

\begin{equation}\label{eq:fista_invar}
v \in \partial_{\eta} \left(  \phi - \frac{\mu}2 \|\cdot-y\|^2 \right) (y), \quad \eta \ge 0, \quad \|v\|^2 + 2\eta \leq \tilde \sigma \|y-x_0\|^2,
\end{equation}
as long as $A_k$ satisfies
\begin{equation}\label{ineq:aux-Ak}
A_k\geq \bar{A}=\bar{A}(\mu,\tilde \sigma ):=\frac{2\mu+1 +\sqrt{(2\mu+1)^2+16\tilde \sigma }}{2\tilde \sigma },
\end{equation}
which in turn is satisfied 
in at most
$$
\left\lceil\min \left\{2\sqrt{(L_f - \mu_f)\bar{A}},\; \left(\frac{1}{2}+\sqrt{\frac{L_f - \mu_f}{\mu}}\right)\log_1^+\left([L_f - \mu_f]\bar{A}\right)+1\right\}\right\rceil
$$
iterations of S-FISTA.
\end{lemma}
\begin{proof}
The first two relations in \eqref{eq:fista_invar} follow immediately from Lemma~\ref{lm:stat-sfista}.
Now, it follows from Lemma~\ref{lm:stat-sfista}(c) and Lemma~\ref{lem:gamma-sfista}(e) that
	\begin{equation}\label{ineq:aux100}
	\|v_k\|^2+2\eta_k \le \left[\frac{2}{A^2_k}\left(1+\tau_k\right)+ \frac{1}{A_k}\right]\|y_k-x_0\|^2= \frac{1}{A_k}\left(\frac{4}{A_k} + 2\mu +1\right)\|y_k-x_0\|^2.
	\end{equation}
Since	
	\[
	\frac{1}{A_k}\left(\frac{4}{A_k} + 2\mu +1\right) \leq \tilde \sigma\Longleftrightarrow \tilde \sigma A_k^2-(2\mu+1)A_k-4 \geq 0,
	\]
we then conclude that the last inequality in \eqref{eq:fista_invar} follows from \eqref{ineq:aux100} and the fact that  the right hand side of \eqref{ineq:aux-Ak} corresponds to the largest root of the above quadratic equation.
The last statement of the lemma follows immediately from  the last statement of Lemma~\ref{lm:Ak-est-fista}. 
\end{proof}

The next result gives the complexity bound for a (slightly) different relative error criterion.

\begin{lem}
Let $\{(y_{k},v_{k},\eta_{k})\}$ and $\bar{A}(\cdot,\cdot)$ be as
in Lemma~\ref{lem:fista_basic_compl}. Then for any $\sigma>0$ and
$k\geq1$, the triple $(y,v,\eta)=(y_{k},v_{k},\eta_{k})$ satisfies
\begin{equation}
v\in\partial_{\eta}\left(\phi-\frac{\mu}{2}\|\cdot-y\|^{2}\right)(y),\quad\eta\geq0,\quad\|v\|^{2}+2\eta\leq\sigma\|v+y-y_{0}\|^{2},\label{eq:alt_fista_invar}
\end{equation}
as long as $A_{k}\geq\bar{A}(\mu,\sigma/(1+\sqrt{\sigma})^{2})$,
which in turn is satisfied in at most
\begin{equation}
\left\lceil \min\left\{ 2\sqrt{(L_f-\mu_{f}){\cal A}_{\mu,\sigma}},\left(\frac{1}{2}+\sqrt{\frac{L_f-\mu_{f}}{\mu}}\right)\log_{1}^{+}\left([L_f-\mu_{f}]{\cal A}_{\mu,\sigma}\right)+1\right\} \right\rceil \label{eq:spec_S-FISTA_bd}
\end{equation}
iterations of S-FISTA, where ${\cal A}_{\mu,\sigma}:=(2\mu+3)(1+\sqrt{\sigma})^{2}/\sigma$.
\end{lem}

\begin{proof}
Let $\tilde{\sigma}=\sigma/(1+\sqrt{\sigma})^{2}\in(0,1)$. Using
Lemma~\ref{lem:fista_basic_compl} and the fact that $y_{0}=x_{0}$,
it follows that first two relations in \eqref{eq:alt_fista_invar}
hold and $\|v\|^{2}+2\eta\leq\tilde{\sigma}\|y-y_{0}\|^{2}.$ Using
the previous inequality, the definition of $\tilde{\sigma}$, and
the relation $(a+b)^{2}\leq(1+\sqrt{\sigma})a^{2}+(1+1/\sqrt{\sigma})b^{2}$ for any $a,b\in{\mathbb R}$,
we have
\[
\|v\|^{2}+2\eta\leq\frac{\sigma}{1+\sqrt{\sigma}}\|v+y-y_{0}\|^{2}+\frac{\sqrt{\sigma}}{1+\sqrt{\sigma}}\|v\|^{2}
\]
which easily implies the last relation in \eqref{eq:alt_fista_invar}.
Finally, to obtain the bound in \eqref{eq:spec_S-FISTA_bd}, we first
use the definitions of $\tilde{\sigma}$ and $\bar{A}_{\mu,\sigma}$
with the fact $\tilde{\sigma}\in(0,1)$ to bound
\[
\bar{A}(\mu,\tilde{\sigma})\leq\frac{2\mu+1}{\tilde{\sigma}}+\frac{2}{\sqrt{\tilde{\sigma}}}\leq\frac{2\mu+3}{\tilde{\sigma}}=\frac{(2\mu+3)(1+\sqrt{\sigma})^{2}}{\sigma}={\cal A}_{\mu,\sigma}.
\]
The iteration complexity now follow from the last statement of Lemma~\ref{lm:Ak-est-fista}
and the above bound.
\end{proof}



The next result shows some bounds on the sequences $\{x_k\}$ and $\{y_k\}$ generated by S-FISTA.

\begin{lemma}\label{lem:aux000}
For every $k \ge 1$, the following estimates hold:
\[
\|x_k-x_0 \| \le \left( \frac{1}{\sqrt{\tau_k}} + 1 \right) d_0, \quad \|y_k-x_0\| \le 2 \left( 1 + \frac{2}{A_k \mu} \right) d_0,\\
\]
\end{lemma}

\begin{proof}
Let $x^*$ be a solution of \eqref{OptProb} such that $\|x_0-x^*\|=d_0$. It follows from Lemma~\ref{lm:hysub-5-ac} with $x=y_k$ and $x=x^*$ that
\begin{equation}\label{aux90}
\tau_k \|x_k-y_k\|^2\leq \|x_0-y_k\|^2, \quad \tau_k \|x_k-x^*\|^2\leq \|x_0-x^*\|^2.
\end{equation}
Hence,  using the triangle inequality for norms, we have
\[
\|x_0-x_k\| \le \|x_0-x^*\| + \|x_k-x^*\| \le \left(1 + \frac1{\sqrt{\tau_k}} \right) \|x_0-x^*\| = \left(1 + \frac1{\sqrt{\tau_k}} \right) d_0,
\]
which proves the first inequality of the lemma.
Moreover, using the triangle inequality for norms  and the first inequality in \eqref{aux90}, we have
\[
\|y_k-x_0\| \le \|x_0-x_k\| + \|x_k-y_k\| \le \|x_0-x_k\| + \frac1{\sqrt{\tau_k}}\|x_0-y_k\|.
\]
Rewriting the above inequality and using the first inequality of the lemma, we have
\[
\left(1 - \frac1{\sqrt{\tau_k}} \right) \|x_0-y_k\| \le \|x_0-x_k\| \le \left(1 + \frac1{\sqrt{\tau_k}} \right) d_0.
\]
Thus,
\[
\|x_0-y_k\| \le \frac{\sqrt{\tau_k} + 1}{\sqrt{\tau_k} - 1} d_0 = \frac{(\sqrt{\tau_k} + 1)^2}{\tau_k - 1} d_0
\le  \frac{2(\tau_k + 1)}{\tau_k-1} d_0 = 2 \left( 1 + \frac2 {\tau_k -1} \right) d_0.
\]
The second inequality of the lemma now follows from the fact that $\tau_k=1+A_k\mu$ in view of Lemma~\ref{lem:gamma-sfista}(e).
\end{proof}
The below result  establishes some alternative iteration complexity bounds for the residuals pair $(v_k,\eta_k)$ defined in \eqref{def:vk-etak}.
\begin{lemma} The following inequalities hold
\begin{equation}\label{ineq:estvk-etak}
\|v_k\| \le \frac{2}{A_k}\left(2+\sqrt{\mu A_k}\right)\left( 1 + \frac{2}{A_k \mu} \right) d_0, \quad \eta_k \le\frac{2}{A_k}\left( 1 + \frac{2}{A_k \mu} \right)^2 d_0^2.
\end{equation}
As a consequence, for given a given tolerance pair $(\varepsilon,\eta)\in \R^2_{++}$,  we have 
\begin{equation}\label{ineq:vk-etak-erroabs}
\|v_k\|\leq \varepsilon, \qquad \eta_k\leq \eta    
\end{equation}
in at most 
\begin{align*}
k & := \left \lceil\min \left\{ 8\left(\frac{1}{\sqrt{\varepsilon}}+\frac{\sqrt{\mu d_0}}{\varepsilon}+\frac{\sqrt{d_0}}{\sqrt{\eta}}\right)\sqrt{\mathcal{M} d_0},\right. \right. \\ 
& \qquad \qquad \enskip \left. \left. \left[\frac{1}{2}+\sqrt{\frac{L_f - \mu_f}{\mu}}\right]\log^+_1\left( 16
\left[\frac{1}{\varepsilon}+\frac{\mu d_0}{\varepsilon^2} +\frac{d_0}{\eta}\right]
\mathcal{M} d_0\right)+1\right\}\right\rceil
\end{align*}
iterations, where
$$
\mathcal{M}=\mathcal{M}(\mu_f,\mu,L_f):=\left(1+\frac{8(L_f - \mu_f)}{\mu}\right)^2(L_f - \mu_f)
$$
\end{lemma}
\begin{proof}
The inequalities in \eqref{ineq:estvk-etak} follows by combining Lemma~\ref{lem:gamma-sfista}(e), Lemma~\ref{lm:stat-sfista}(c), and  Lemma~\ref{lem:aux000}.

Now, in view of \eqref{eq:Akest-sfista}, we have $A_k\geq A_1 \geq 1/[4(L_f - \mu_f)]$ for every $k\geq 1$. Hence,   it follows from \eqref{ineq:estvk-etak} that $$
\|v_k\| \le \frac{2}{A_k}\left(2+\sqrt{\mu A_k}\right)\left( 1 + \frac{8(L_f - \mu_f)}{\mu} \right) d_0, \quad \eta_k \le\frac{2}{A_k}\left( 1 + \frac{8(L_f - \mu_f)}{\mu} \right)^2 d_0^2
$$
which implies that in order to $(v_k,\varepsilon_k)$ to satisfy  \eqref{ineq:vk-etak-erroabs}, it is sufficient to have
\begin{align*}
\begin{gathered}
\frac{4}{A_k}\left( 1 + \frac{8(L_f - \mu_f)}{\mu} \right) d_0 \leq \frac{\varepsilon}{2}, \qquad
\frac{2\sqrt{\mu}}{\sqrt{A_k}}\left( 1 + \frac{8(L_f - \mu_f)}{\mu} \right) d_0 \leq \frac{\varepsilon}{2}, \\
\frac{2}{A_k}\left( 1 + \frac{8(L_f - \mu_f)}{\mu} \right)^2 d_0^2 \leq \eta.
\end{gathered}
\end{align*}
Note that the above inequalities are satisfied if
$$
A_k\geq \frac{8}{\varepsilon}\left( 1 + \frac{8(L_f - \mu_f)}{\mu} \right) d_0+\left(\frac{16\mu}{\varepsilon^2}+\frac{2}{\eta}\right)\left( 1 + \frac{8(L_f - \mu_f)}{\mu} \right)^2 d_0^2.
$$
Hence, the last statement of the lemma follows from the above inequalities,  the last 
statement of Lemma~\ref{lm:Ak-est-fista}, and the definition of $\mathcal{M}$.
\end{proof}

\section{Alternate Formulations of S-FISTA}\label{Sec:Fista-mu=0}
This section presents alternate formulations of S-FISTA for the case of $\mu=0$. Although, we assume that $\mu=0$, it is worth mentioning that similar results as the ones obtained in this section can be extended for the general case where $\mu\geq 0$.


We begin by deriving an alternate expression for $y_{k+1}$.

\begin{lemma} \label{lm:ty=y-fista} Assume that $\mu=0$. Then,
for every $k \ge 0$, we have
\[
y_{k+1}= \frac{A_k y_k+a_k x_{k+1}}{A_{k+1}}.
\]
\end{lemma}

\begin{proof}
It follows from \eqref{eq:xnext-sfista} and \eqref{eq:Aalam-sfista} that
\[
\frac{a_k}{A_{k+1}} (x_{k+1}-x_k)  =  \frac{\lam}{a_k} (x_{k+1}-x_k) = y_{k+1}-\tx_k.
\]
On the other hand, it follows from the last identity in \eqref{def:ak-sfista} that
\[
\frac{a_k}{A_{k+1}} (x_{k+1}-x_k)  =  \frac{A_k y_k+a_k x_{k+1}}{A_{k+1}} - \tx_k.
\]
The result now follows by combining the above two identities.
\end{proof}

The next result shows that  the auxiliary sequence  $\{\tilde x_k\}$ generated by S-FISTA can be expressed in terms of the sequence $\{y_k\}$ and a scalar sequence that can be easily generated by solving  a quadratic equation. 
\begin{lemma} \label{lem:sfista_fista} Assume that $\mu=0$ and, for every $k \ge 0$, define
\beq \label{def:tk-fista}
t_k := \frac{A_{k+1}}{a_k} = \frac{a_k}\lam.
\eeq 
Then, for every $k \ge 0$, we have
\begin{equation}\label{rel:xtildek-tk}
\tx_{k+1}= y_{k+1} + \frac{t_k-1}{t_{k+1}} (y_{k+1}-y_k)    
\end{equation}
and
\begin{equation}\label{relation:tk}
t_{k+1}^2 - t_{k+1} - t_k^2 =0.
\end{equation}

\end{lemma}

\begin{proof}
First, note the that the second equality of \eqref{def:tk-fista} follows from \eqref{def:ak-sfista}.
It follows from \eqref{def:tk-fista} with $k=k+1$ and the two last identities in \eqref{def:ak-sfista} both with $k=k+1$ that
\[
\tx_{k+1} - y_{k+1} = \frac{A_{k+1}y_{k+1}+a_{k+1}x_{k+1}}{A_{k+2}} - y_{k+1} =
\frac{a_{k+1}}{A_{k+2}} ( x_{k+1} - y_{k+1}) = \frac{1}{t_{k+1}}  ( x_{k+1} - y_{k+1}).
\]
On the other hand, it follows from \eqref{def:tk-fista}, the second identity in \eqref{def:ak-sfista}, and Lemma \ref{lm:ty=y-fista}, that
\begin{align*}
(t_k-1) (y_{k+1}-y_k) &= \left( \frac{A_{k+1}}{a_k} -1 \right) (y_{k+1}-y_k)
= \frac{A_k}{a_k} \left( y_{k+1} - y_k \right) \\
&= \frac{1}{a_k} \left[ A_ky_{k+1} - ( A_{k+1} y_{k+1} - a_k x_{k+1}) \right] = x_{k+1} - y_{k+1}.
\end{align*}
The first identity of the lemma  now follows by combining the above two identities.
Now, it follows from  \eqref{def:tk-fista}  that
\[
t_k^2 = \frac{A_{k+1}}{\lam}
\]
for every $k \ge 0$. The last identity, together with \eqref{def:tk-fista} and  the second identity in \eqref{def:ak-sfista} with $k=k+1$, then imply that
\[
t_{k+1}^2 - t_{k+1} = \frac{A_{k+2}}{\lam} - \frac{a_{k+1}}{\lam} = \frac{A_{k+1}}{\lam} = t_k^2,
\]
and hence that the second identity of the lemma also holds.
\end{proof}

We now make a few remarks about the relations above and how they relate to the ones given in FISTA. First, \eqref{relation:tk} implies that the iterates $\{t_k\}$ have the recursive form
\[
t_{k+1} = \frac{1+\sqrt{1+4t_k^2}}{2}.
\]
Second, in view of the first remark, \eqref{rel:xtildek-tk}, and the fact that $t_0=1$, we conclude that the iterates $\{(y_k,\tilde{x}_k,t_k)\}$ generated by S-FISTA are the same as the ones generated by FISTA (see, for example, the definitions in \cite{beck2017first,beck2009fast}).









The next result presents  an alternative way of expressing the relations in \eqref{rel:xtildek-tk} and \eqref{relation:tk}.

\begin{lemma}
Assume $\mu=0$, let $\{t_k\}$ be as in \eqref{def:tk-fista}, and define $\alpha_k= 1/t_k$ for every $k\geq 0$. Then, the following relation holds
\begin{align}
\alpha_{k+1}^2&=(1-\alpha_{k+1})\alpha_k^2,\label{aux:rel-alphak}\\
\tilde x_{k+1}&=y_{k+1} +\frac{\alpha_k\left(1-\alpha_k\right)}{\alpha_k^2+\alpha_{k+1}}(y_{k+1}-y_k).\label{aux:rel-xtildek-alphak}
\end{align}
\end{lemma}
\begin{proof}
It follows from \eqref{relation:tk} and the definition of $\alpha_k$ that 
$$
\frac{1}{\alpha_{k+1}^2}-\frac{1}{\alpha_{k+1}}-\frac{1}{\alpha_k^2}=0.
$$
Multiplying both sides by $\alpha_k^2\alpha_{k+1}^2$, we arrive at
\begin{equation*}
\alpha_k^2-\alpha_k^2\alpha_{k+1}- \alpha_{k+1}^2=0
\end{equation*}
which immediately implies \eqref{aux:rel-alphak}.
Now, note that \eqref{rel:xtildek-tk} together with the definition of $\alpha_k$ imply that
\begin{align*}
\tilde x_{k+1}  &= y_{k+1} +\frac{t_k-1}{t_{k+1}}(y_{k+1}-y_k)=y_{k+1} +\alpha_{k+1}\left(\frac{1}{\alpha_k}-1\right)(y_{k+1}-y_k)\\
&=y_{k+1} +\frac{\alpha_{k+1}}{\alpha_k}\left(1-\alpha_k\right)(y_{k+1}-y_k),   
\end{align*}
which in view of  \eqref{aux:rel-alphak} proves \eqref{aux:rel-xtildek-alphak}.
\end{proof}

Similar to the remarks after Lemma~\ref{lem:sfista_fista}, the above result  shows that when $\mu=0$ and $h$ is the characteristic function of a simple set, the iterates $\{(y_k,\tilde{x}_k,t_k)\}$ generated by S-FISTA are the same as the ones generated by Nesterov's FGM in \cite[Eq (2.2.63)]{nesterov2018lectures} with $\alpha_0 = 1$ (see also \cite[Eq (2.2.17)]{nesterov2003introductory}).

\bibliographystyle{abbrvnat}
\bibliography{fista_refs}

\end{document}